\documentclass[12pt]{article}

\usepackage{subcaption}

\newcommand{\comment}[1]{}

\usepackage{amssymb}
\usepackage{amsthm}
\usepackage{amsmath}

\usepackage{color}
\definecolor{teal}{RGB}{0,128,128}
\definecolor{darkpurple}{RGB}{128,0,128}
\definecolor{gray}{RGB}{150,150,150}

\theoremstyle{plain}
\newtheorem{theorem}{Theorem}[section]
\newtheorem{lemma}[theorem]{Lemma}
\newtheorem{corollary}[theorem]{Corollary}
\newtheorem{question}[theorem]{Question}

\theoremstyle{definition}
\newtheorem{definition}[theorem]{Definition}
\newtheorem{example}[theorem]{Example}

\newcommand{\cB}{{\cal B}}

\newcommand{\PD}{\mathrm{PD}}
\newcommand{\DPD}{\mathrm{DPD}}
\newcommand{\PDN}{\mathrm{PDN}}
\newcommand{\DPDN}{\mathrm{DPDN}}

\newcommand{\lc}{\left\lceil}
\newcommand{\rc}{\right\rceil}

\newcommand{\lf}{\left\lfloor}
\newcommand{\rf}{\right\rfloor}

\title{
	Packing designs with large block size}
\author{
	Andrea C.~Burgess\thanks{Department of Mathematics and Statistics, University of New Brunswick, Saint John, NB,  Canada.}	 \thanks{Corresponding author.  Email: \texttt{andrea.burgess@unb.ca}.}
	\and
	Peter Danziger\thanks{Department of Mathematics,
	Toronto Metropolitan University, Toronto, ON, Canada.}
	\and
	Daniel Horsley\thanks{School of Mathematics, Monash University, Clayton, VIC, Australia.}
	\and
	Muhammad Tariq Javed\thanks{Department of Computer Science, Toronto Metropolitan University, Toronto, ON, Canada.}
}

\begin{document}
	\maketitle

\begin{abstract}
Given positive integers $v$, $k$, $t$ and $\lambda$ with $v \geq k \geq t$, a {packing design} $\PD_{\lambda}(v,k,t)$ is a pair $(V,\cB)$, where $V$ is a $v$-set and $\cB$ is a collection of $k$-subsets of $V$ such that each $t$-subset of $V$ appears in at most $\lambda$ elements of $\cB$. When $\lambda=1$, a $\PD_1(v,k,t)$ is equivalent to a binary code with length $v$, minimum distance $2(k-t+1)$ and constant weight $k$.  The maximum size of a $\PD_{\lambda}(v,k,t)$ is called the {packing number}, denoted $\PDN_{\lambda}(v,k,t)$.
In this paper we consider packing designs with $k$ large relative to $v$.  We prove that for a positive integer $n$, $\PDN_{\lambda}(v,k,t) = n$ whenever $nk-(t-1)\binom{n}{\lambda+1} \leq \lambda v < (n+1)k-(t-1)\binom{n+1}{\lambda+1}$.

We also prove that if no point appears in more than three blocks, then the blocks of a $\PD_2(v,k,2)$ can be ordered so that no ordered pair occurs more than once. This produces a directed packing design and we show that the corresponding directed packing number is equal to $n$ when $nk-\binom{n}{3} \leq 2v < (n+1)k-\binom{n+1}{3}$. Such directed packing designs yield $(k-t)$-insertion/deletion codes.
\end{abstract}

\section{Introduction}
Given positive integers $v$, $k$, $t$ and $\lambda$ with $v \geq k \geq t$, a {\em packing design} $\PD_{\lambda}(v,k,t)$ is a pair $(V,\cB)$, where $V$ is a $v$-set and $\cB$ is a collection of $k$-subsets of $V$ such that each $t$-subset of $V$ appears in at most $\lambda$ elements of $\cB$.  The elements of $V$ are called {\em points} and the elements of $\cB$ are {\em blocks}.  Packing designs have been extensively studied; the reader is referred to~\cite{MillsMullin} and~\cite{StinsonWeiYin}.

For given integers $v$, $k$, $t$ and $\lambda$, the existence of a $\PD_{\lambda}(v,k,t)$ is trivial; simply take a single $k$-subset of points.  Thus of more interest is the {\em size} of a packing design, defined as its number of blocks.  We denote by $\PD_{\lambda}(n;v,k,t)$ a $\PD_{\lambda}(v,k,t)$ of size $n$.  The maximum size of a $\PD_{\lambda}(v,k,t)$ is called the {\em packing number} $\PDN_{\lambda}(v,k,t)$.  The case where $t=2$ has received particular attention in the literature, and in this case we remove $t$ from the notation, and write $\PD_{\lambda}(v,k)$, $\PD_{\lambda}(n;v,k)$ and $\PDN_{\lambda}(v,k)$.  If $\lambda=1$, we also drop it from the notation.

\begin{example}
The following blocks form a $\PD(4;6,3)$ on the point set $\{0,\ldots,5\}$.
\[
\begin{array}{l}
\{0,1,2\} \\
\{0,3,4\} \\
\{1,3,5\} \\
\{2,4,5\}
\end{array}
\]
\end{example}

When each $t$-subset of $V$ appears in {\em exactly} $\lambda$ blocks, a $\PD_{\lambda}(v,k,t)$ is called a {\em $t$-$(v,k,\lambda)$ design}, or a {\em balanced incomplete block design} $\mathrm{BIBD}(v,k,\lambda)$ when $t=2$.  These designs have been the subject of extensive study, and the reader is referred to~\cite{Handbook} for further details.  

Johnson~\cite{Johnson} and Sch\"{o}nheim~\cite{Schonheim} gave the following bound for packings:
\begin{equation} \label{Schonheim bound}
\PDN_{\lambda}(v,k,t) \leq U_{\lambda}(v,k,t),
\end{equation}
where
\[
U_{\lambda}(v,k,t) = \left\lfloor \frac{v}{k} \left\lfloor \frac{v-1}{k-1} \left\lfloor \cdots \left\lfloor \frac{v-t+2}{k-t+2} \left\lfloor \frac{\lambda (v-t+1)}{k-t+1} \right\rfloor \right\rfloor \cdots \right\rfloor \right\rfloor \right\rfloor.
\]
Hanani \cite{Hanani} improved the Johnson-Sch\"{o}nheim bound in the case $t=2$. In particular, he showed that when
\begin{equation}
\label{U-1 Condition}
\lambda(v-1) \equiv 0 \pmod{k-1} \quad\mbox{ and }\quad \lambda v(v-1) \equiv -1 \pmod{k},
\end{equation}
then $\PDN_{\lambda}(v,k) \leq U_{\lambda}(v,k,2)-1$.
As a result, we define
\[
B_\lambda(v,k) = \left\{\begin{array}{ll}
	U_{\lambda}(v,k,2)-1 & \mbox{if $\lambda$, $k$ and $v$ satisfy \eqref{U-1 Condition}}; \\
	U_{\lambda}(v,k,2) & \mbox{otherwise,}
\end{array}\right.
\]
and we can say that $\PDN_{\lambda}(v,k) \leq B_\lambda(v,k)$.
For further details, see~\cite{MillsMullin}.

In the case $\lambda=1$, Johnson~\cite{Johnson} proved another bound on the packing number, known as the {\em second Johnson Bound}.  Letting $d=\PDN(v,k,t)$, $d(d-1)(t-1) \geq q(q-1)v+2qr$, where $q$ and $r$ are the integers satisfying $kd=qv+r$ with $0 \leq r < v$.
When $v(t-1)<k^2$, this implies a bound that is slightly weaker but is easier to implement, namely:
\[
\PDN(v,k,t) \leq \lf \frac{v(k+1-t)}{k^2-v(t-1)} \rf.
\]

More recently, the third author \cite{horsley_bounds}, by generalizing Bose's proof of Fisher's inequality, gave new bounds on packing numbers for parameter sets satisfying $\lambda (v-1) < k(k-1)$.

\begin{theorem}[\cite{horsley_bounds}]\label{th:horsley_pack_1}
Given positive integers $v$, $k$ and $\lambda$ such that $3 \leq k < v$, let $r$ and $d$ be the integers with $0 \leq d < k-1$ such that $\lambda(v-1) = r(k-1)+d$.  If $d < r-\lambda$, then
	\[
	\PDN_{\lambda}(v,k)  \leq \lf \frac{v(r-1)}{k-1} \rf.
	\]
\end{theorem}
In~\cite{horsley_bounds}, it is noted that this bound is less than or equal to $U_{\lambda}(v,k,2)$ whenever $r<k$.  However, it is a weaker bound than the second Johnson bound when $\lambda=1$. He also proved the following for the case $d \geq r-\lambda$.

\begin{theorem}[\cite{horsley_bounds}]\label{th:horsley_pack_2}
Given positive integers $v$, $k$ and $\lambda$ such that $3 \leq k < v$, let $r$ and $d$ be the integers with $0 \leq d < k=1$ such that $\lambda(v-1)=r(k-1)+d$.  If $d \geq r-\lambda$, then
\[
\PDN_{\lambda}(k,v) \leq \lf \frac{rv(\alpha-\beta)-\alpha v}{k(\alpha-\beta)-1}\rf,
\]
where:
\begin{enumerate}
\item $\alpha=\frac{r-\lambda+1}{2d+2}$ and $\beta=0$ if $k(r+\lambda-1)>2d+2$;
\item $\alpha=\frac{r-\lambda}{2d+2}$ and $\beta=\frac{r-\lambda}{2(d+k)}$ if $(r-\lambda)k(k-1)>2(d+1)(d+k)$.
\end{enumerate}
\end{theorem}

The packing numbers $\PDN(v,k)$ for $k \in \{3,4\}$ are known, many values for $k=5$ are known and some for $k=6$. See~\cite{MillsMullin} and the update \cite{StinsonWeiYin} for details. We note that for $k=3$ and all $v \geq 3$, $\PDN_{\lambda}(v,3) = B_{\lambda}(v,3)$.  For $k=4$, $\PDN_{\lambda}(v,k) = B_{\lambda}(v,k)$, except when $\lambda=1$ and $v \equiv 7$ or $10 \pmod{12}$ (in which case $\PDN(v,k) = B(v,k)-1$ if $v \notin \{10,19\}$), or $(\lambda,v) = \{(1,8), (1,8), (1,11), (1,17), (2,9), (3,6)\}$.
It is well known that a $\PD(n; v,k,t)$ is equivalent to a binary code of length $v$, size $n$, constant weight $k$, and minimum Hamming distance at least $2(k - t + 1)$, see \cite{Tonchev}.

In this paper, we also study a variant of packing designs in which the order of elements in the blocks matters.  A {\em directed packing design}, $\DPD_{\lambda}(v,k,t)$, is a pair $(V,\cB)$, where $V$ is a set of $v$ points and $\cB$ is a collection of ordered $k$-tuples of distinct elements of $V$, called {\em blocks} such that each $t$-tuple of elements of $V$ appears in at most $\lambda$ blocks. Note that for an ordered $t$-tuple to appear in a block $B$, it must form a subsequence of $B$ but its elements need not be consecutive in $B$; thus, for example, the ordered triple $(0,2,4)$ appears in the $k$-tuple $(0,1,2,3,4)$.

As in the undirected case, the question of interest is the maximum size of a directed packing design.  We denote a directed packing design with $n$ blocks by $\DPD_{\lambda}(n;v,k,t)$, and define the {\em directed packing number} $\DPDN_{\lambda}(v,k,t)$ to be the maximum value of $n$ for which a $\DPD_{\lambda}(n;v,k,t)$ exists. Again, when $t=2$ we remove $t$ from the notation and write $\DPD_{\lambda}(v,k)$, $\DPD_{\lambda}(n;v,k)$ and $\DPDN_{\lambda}(v,k)$. When $\lambda=1$, we also remove it from the notation. For directed packing designs our primary focus is on the case where $(t,\lambda)=(2,1)$.

\begin{example}
The following blocks form a $\DPD(4;6,4)$ on the point set $\{0,\ldots,5\}$.
\[
\begin{array}{l}
(0,1,2,3) \\
(4,3,5,0) \\
(5,3,2,4) \\
(2,1,0,5)
\end{array}
\]
\end{example}

Note that since each $t$-tuple has $t!$ distinct orderings, replacing each $k$-tuple of a $\DPD_{\lambda}(v,k,t)$ with the $k$-set of its coordinates results in a $\PD_{t!\lambda}(v,k,t)$.  This fact gives us the following immediate bound on the directed packing number.
\begin{lemma} \label{directed vs undirected}
If $v \geq k \geq t$ and $\lambda$ are positive integers, then $\DPDN_{\lambda}(v,k,t) \leq \PDN_{t!\lambda}(v,k,t)$.  In particular, $\DPDN(v,k) \leq \PDN_{2}(v,k)$.
\end{lemma}

As a consequence of Lemma~\ref{directed vs undirected} we have
\[
\DPDN_{\lambda}(v,k,t) \leq U_{t!\lambda}(v,k,t) \mbox{ and } \DPDN_{\lambda}(v,k) \leq B_{2\lambda}(v,k).
\]

When every ordered pair of points appears exactly once, a $\DPD_\lambda(v,k)$ is known as a directed balanced incomplete block design, denoted DBIBD$(v,k,\lambda)$.
In order for this to happen, it is not hard to see that we require that
\begin{equation}\label{DBIBD Nec}
2\lambda(v-1)\equiv 0 \pmod{k-1} \mbox{ and } 2\lambda v(v-1)\equiv 0 \pmod{k(k-1)}.
\end{equation}
In \cite{ChangGiovanni1, ChangGiovanni2} asymptotic existence of DBIBD$(v,k,\lambda)$ is shown. Specifically, they show that for $v$ and $k$ satisfying the necessary conditions \eqref{DBIBD Nec}, a DBIBD$(v,k,\lambda)$ exists for all $v > \exp(k^{3k^6})$ when $k\equiv 0, 2,3\pmod 4$, and for $v > \exp(\exp(k^{6k^2}))$ when $k\equiv 1 \pmod 4$.
Additionally, DBIBDs with small block size have been studied \cite{AbelBennett, Bennet, SeberrySkillicorn, StreetSeberry, StreetWilson, WangYin}. 
For $k = 3,4,5, 6,7$ we have the following theorem, see \cite{BennettMahmoodi}, and \cite{AbelBennett}.
\begin{theorem}
	A {\rm DBIBD}$_\lambda(v, k,\lambda)$ exists for all $v$, $\lambda$ and $k \in \{3, 4, 5, 6\}$ if and only if $2\lambda(v-1)\equiv 0 \pmod{k-1}$  and $2\lambda v(v-1)\equiv 0 \pmod{k(k-1)}$, except when $(v, k, \lambda) = (15, 5, 1)$ or $(21, 6, 1)$.

	A {\rm DBIBD}$(v, 7, 1)$ exists if and only if $v \equiv 1 \pmod 3$ and $v(v - 1) \equiv 0 \pmod{21}$, with the exception of $v = 22$ and the possible exceptions of $v \in \{274, 358\}$.
\end{theorem}

More generally, directed packing designs with $(t,\lambda)=(2,1)$ and small block size have also been studied and we have the following theorem.
\begin{theorem}[\cite{AbelEtAl, AssafShalabyYin2, AssafShalabyYin, AssafShalabyMahmoodiYin,  ShalabyYin, Skillicorn k=4, Skillicorn k=3}] \mbox{}
	For $t=2$ and $k \in \{3,4,5\}$, 
	$\DPDN_\lambda(v,k) = B_{2\lambda}(v,k)$, except that
	\begin{itemize}
	\item $\DPDN(v,k) = U_2(v,k,2)-1$ for  $(v,k) \in \{(9,4), (13,5)\}$,
	\item $\DPDN(15,5) = U_2(15,5,2)-2$,
	\end{itemize}
	and except possibly when $k=5$ and $(v,\lambda) \in \{(19,1), (27,1), (43,3)\}$.
\end{theorem}
Directed designs are well known to be connected to insertion/deletion-correcting codes \cite{LevenshteinRussian}, see also \cite{LevenshteinHandbook}. In particular, given a $\DPD(v,k,t)$, $(V, \cB)$, taking the alphabet to be $V$ and the codewords as the blocks of the design, $\cB$, we obtain a $(k-t)$-insertion/deletion code. Specifically, since each $t$-subsequence only appears in a single codeword, if $(k-t)$ symbols are changed, we cannot reach another codeword.
Much of the literature connecting directed packing designs to insertion/deletion codes concentrates on perfect insertion/deletion codes, which correspond to DBIBDs \cite{Bours, Klein, Levenshtein, WangYin}. However, we will see that our results determine the maximum size of certain insertion/deletion codes, even in cases where the corresponding directed design cannot exist because the necessary conditions \eqref{DBIBD Nec} are not satisfied.

With reference to the correspondence between (undirected) packing designs and constant-weight codes, we are considering large values of $k$ relative to $v$, which by the Johnson-Sch\"onheim bound limits the number of possible codewords. However, in this case the minimum distance in a constant-weight code arising from a $\PD(v,k,t)$ is $2(k-t+1)$, and hence the number of errors which may be detected/corrected is large.
Similarly, the number of substitutions detected, $k-2$, in an insertion/deletion code arising from a $\DPD(v,k)$ is large, which may be desirable in particular situations.
In fact, we show in these cases that the maximum possible number of codewords is much smaller than the Johnson-Sch\"onheim bound.

The remainder of the paper proceeds as follows.  In Section~\ref{Sec:Prelim}, we describe some basic properties of packings which are useful in determining packing numbers.  In Section~\ref{Sec:Main}, we determine the value of $\PDN_{\lambda}(v,k,t)$ when $k$ is bounded below by $\frac{\lambda v}{n+1} + \frac{t-1}{\lambda+1}\binom{n}{\lambda}$ for some positive integer $n$.  Section~\ref{sec:DPN} extends a special case of this result to directed packings.  Finally, in Section~\ref{Sec:Conclusion}, we summarize our results, discuss the performance of known bounds,  and state a conjecture regarding the value of $\DPDN_{\lambda}(v,k)$.

\section{Preliminaries} \label{Sec:Prelim}

In this section, we give some preliminary results regarding packings, including bounds on the number of times elements may occur within a (directed) packing design.

\begin{definition}
Let $(V,\cB)$ be a $\PD_{\lambda}(v,k,t)$ or a $\DPD_{\lambda}(v,k,t)$, and let $x \in V$.  The {\em frequency} of $x$, denoted
$r_{\cB}(x)$, is the number of blocks in $\cB$ in which $x$ occurs.  If the set $\cB$ of blocks is clear from context, we denote the frequency by
$r(x)$.
\end{definition}

In the case that every $t$-set of points appears in {\em exactly} $\lambda$ blocks, so that we have a $t$-design, then each point has the same frequency, i.e.\ for all $x \in V$, $r(x)=r$, the replication number of the design.  

Straightforward counting gives the following well-known bound on
$r(x)$.
\begin{lemma}
If $(V,\cB)$ is a $\PD_{\lambda}(v,k,t)$ and $x \in V$, then $
{r(x)} \leq \lf \lambda\binom{v-1}{t-1}/\binom{k-1}{t-1} \rf$.
\end{lemma}
It follows that $r(x) \leq \lf \frac{2\lambda(v-1)}{k-1}\rf$ for any point $x$ in a $\DPD_{\lambda}(v,k)$, since removing the direction from the blocks of such a design results in a $\PD_{2\lambda}(v,k)$.

We now define some useful notation that we will use throughout.
Let $(V,\cB)$ be a $\PD_{\lambda}(n;v,k,t)$ or a $\DPD_{\lambda}(n;v,k,t)$. Given a point $x\in V$, we will use $\cB_x$ to denote the set of blocks that contain $x$; note that $r_{\cB}(x) =|\cB_x|$. 
For each $i \in \{0,\ldots,n\}$, let $V_i$ be the set of points with frequency $i$ and $N_i=|V_i|$ be the number of points with frequency $i$.  By counting the number of distinct points, as well as the number of points in the blocks with repetition, we get the following equations:
\begin{eqnarray}
\sum_{i=0}^n N_i &=& v \label{MillsEquation1} \\
\sum_{i=1}^{n} i N_i &=& nk \label{MillsEquation2}
\end{eqnarray}

The following lemma gives an inequality relating the values of the $N_i$.
\begin{lemma}\label{General N-bound}
In a $\PD_{\lambda}(n;v,k,t)$,
\[
\sum_{i=1}^n \binom{i}{\lambda+1} N_i \leq (t-1)\binom{n}{\lambda+1}.
\]
\end{lemma}

\begin{proof}
If $n \leq \lambda$, then the result is trivial.  Otherwise, each $(\lambda + 1)$-set of blocks from
$\mathcal{B}$ can be a subset of $\mathcal{B}_x$
for at most $t-1$ points $x \in V$, for otherwise $t$ points would appear together in more than $\lambda$ blocks. The number of $(\lambda + 1)$-subsets of $\mathcal{B}$ is $\binom{n}{\lambda+1}$ and, for each integer $i \geq \lambda + 1$, the number of $(\lambda + 1)$-subsets of $\mathcal{B}_x$ for each
point $x \in V_i$ is $\binom{i}{\lambda+1}$.
Thus $\sum_{i=\lambda+1}^n \binom{i}{\lambda+1} N_i \leq (t-1)\binom{n}{\lambda+1}$ and the result follows.
\end{proof}

We note, in addition, that there cannot be too many points of high frequency.

\begin{lemma}
In any $\PD_{\lambda}(n;v,k,t)$, no $t$ points can have frequencies with sum greater than $(t-1)n+\lambda$.  In particular, there can be at most $t-1$ points with frequency greater than $((t-1)n+\lambda)/t$.
\end{lemma}

\begin{proof}
Suppose $T$ is a $t$-set of points in a $\PD_{\lambda}(n;v,k,t)$. For each $x \in T$, there are $n-r(x)$ blocks that $x$ does not occur in and so in total there are at most $tn-\sum_{x \in T}r(x)$ blocks that are not supersets of $T$. By the definition of a $\PD_{\lambda}(n;v,k,t)$, at most $\lambda$ blocks can be supersets of $T$ and hence at least $n-\lambda$ blocks are not. So $tn-\sum_{x \in T}r(x)  \geq n-\lambda$ and rearranging gives the result.
\end{proof}

We close this section by considering how the maximum size of a (directed) packing may change when the number of points or the block size changes.  Note that the blocks of a $\PD_{\lambda}(n;v,k,t)$ (resp.\ $\DPD_{\lambda}(n;v,k,t)$) form a $\PD_{\lambda}(n;v',k,t)$ (resp.\ $\DPD_{\lambda}(n;v',k,t)$) for any $v' \geq v$.  Moreover, if we remove a fixed number of elements of each block of a $\PD_{\lambda}(n;v,k,t)$ (resp.\ $\DPD_{\lambda}(n;v,k,t)$) we obtain a $\PD_{\lambda}(n;v,k',t)$ (resp.\ $\DPD_{\lambda}(n;v,k',t)$) for $k' \leq k$.  These observations give us the following.
\begin{theorem} \label{add-delete points}
Suppose $v \geq k \geq t$ and $\lambda$ are positive integers, and let $v'$ and $k'$ be integers with $t \leq k' \leq k$ and $v \leq v'$.  Then:
\begin{enumerate}
\item $\PDN_{\lambda}(v,k,t) \leq \PDN_{\lambda}(v',k,t)$,
\item $\PDN_{\lambda}(v,k,t) \geq \PDN_{\lambda}(v,k',t)$,
\item $\DPDN_{\lambda}(v,k,t) \leq \DPDN_{\lambda}(v',k,t)$, and
\item $\DPDN_{\lambda}(v,k,t) \geq \DPDN_{\lambda}(v,k',t)$.
\end{enumerate}
\end{theorem}

\section{Packing numbers for large block size} \label{Sec:Main}

The main result in this section is Theorem~\ref{PDN}, which determines the value of $\PDN_{\lambda}(v,k,t)$ when the parameters satisfy $kn-(t-1)\binom{n}{\lambda+1} \leq \lambda v < k(n+1)-(t-1)\binom{n+1}{\lambda+1}$ for some nonnegative integer $n$. For fixed $t$ and $\lambda$, this is the case when  $k$ is large with respect to $v$.
Note that in expressions of the form $\binom{a}{b}$, we use the convention that if $a<b$, then $\binom{a}{b}=0$.
We begin by proving a result which generalises the second Johnson bound to packing designs with arbitrary values of $\lambda$.

\begin{theorem}\label{General Upper Bound}
Let $v \geq k \geq t$ and $\lambda$ be positive integers. If there exists a $\PD_{\lambda}(d;v,k,t)$, then
\[(t-1)\binom{d}{\lambda+1} \geq v\binom{q}{\lambda+1}+r\binom{q}{\lambda}\]
where $q$ and $r$ are the integers such that $dk=qv+r$ and $0 \leq r <v$.
\end{theorem}

\begin{proof}
Consider a $\PD_{\lambda}(d;v,k,t)$. For this packing design, by Lemma~\ref{General N-bound} we have
\begin{equation}\label{E:NBoundRestatement}
(t-1)\binom{d}{\lambda+1} \geq \sum_{i=1}^{n} N_i\binom{i}{\lambda+1}.
\end{equation}
Since $\sum_{i=1}^{d} N_i=v$, $\sum_{i=1}^{d} iN_i=dk$ and the function $\binom{i}{\lambda+1}$ is convex in $i$, we have that the right hand side of \eqref{E:NBoundRestatement} is minimized when $N_{q+1}=r$, $N_q=v-r$, and $N_i=0$ for all $i \in \{1,\ldots,n\} \setminus \{q,q+1\}$. Thus
\[
(t-1)\binom{d}{\lambda+1} \geq r\binom{q+1}{\lambda+1}+(v-r)\binom{q}{\lambda+1}=v\binom{q}{\lambda+1}+r\binom{q}{\lambda}
\]
where the equality follows by substituting $\binom{q+1}{\lambda+1}=\binom{q}{\lambda+1}+\binom{q}{\lambda}$ and simplifying.
\end{proof}

Note that, when $\lambda=1$, Lemma~\ref{General Upper Bound} reduces to the classical second Johnson bound for packing designs. We will show that in certain cases the bound of Lemma~\ref{General Upper Bound} is indeed met by constructing a $\PD_{\lambda}(n;v,k,t)$ for values of the parameters satisfying the hypothesis.  We first require the following lemma.

\begin{lemma} \label{nk/v Lemma}
For any $v \geq k \geq t \geq 1$ and $n \geq 0$, there is a $\PD_{\lc nk/v \rc}(n;v,k,t)$, $(V,\mathcal{B})$, such that $
r(x)
\in \left\{ \lf \frac{nk}{v} \rf, \lc \frac{nk}{v} \rc\right\}$ for all $x \in V$.
\end{lemma}

\begin{proof}
If $v=k$, then the multiset consisting of $n$ copies of $V$ forms the blocks of the required packing design, so we may assume $v>k$.  Fix $v$ and $k$, and proceed by induction on $n$. The result is trivial for $n=0$ and $n=1$.  Otherwise, take a $\PD_{\lc nk/v \rc}(n;v,k,t)$, $(V,\mathcal{B})$, such that $r_{\cB}(x)
\in \left\{ \lf \frac{nk}{v} \rf, \lc \frac{nk}{v} \rc \right\}$ for all $x \in V$.  Let $\mathcal{B}' = \mathcal{B} \cup \{B\}$, where $B$ is a $k$-subset of $V$ such that $B \subseteq V_{\lf nk/v \rf}$ if $N_{\lf nk/v \rf} \geq k$, and $V_{\lf nk/v \rf} \subseteq B$ if $N_{\lf nk/v\rf} < k$.

Then we have that $|r_{\cB'}(x)-r_{\cB'}(y)| \leq 1$ for all $x,y \in V$. Thus, since $\sum_{x \in V}r_{\cB'}(x)=(n+1)k$, we have $r_{\cB'}(x) \in \{ \lfloor \frac{(n+1)k}{v} \rfloor, \lceil \frac{(n+1)k}{v} \rceil \}$. Since any point is in at most $\lceil \frac{(n+1)k}{v} \rceil$ blocks, the same is certainly true of any set of $t$ points, and hence $(V,\mathcal{B}')$ is a $\PD_{\lc (n+1)k/v \rc}(n+1;v,k,t)$.
\end{proof}

\begin{lemma} \label{General Construction}
Let $v \geq k \geq t \geq 2$, $\lambda \geq 1$ and $n \geq 1$ be integers.  If $k \geq (t-1)\binom{n-1}{\lambda}$
and  $\lambda v \geq nk-(t-1)\binom{n}{\lambda+1}$,
then there is a $\PD_{\lambda}(n;v,k,t)$, $(V,\mathcal{B})$, such that $r(x) 
\leq \lambda+1$ for all $x \in V$.
\end{lemma}

\begin{proof}
If $n \leq \lambda$, then taking $\mathcal{B}$ to be any collection of $n$ $k$-subsets of $V$ gives a packing satisfying the required conditions.  So we may assume that $n \geq \lambda+1$. Note that by combining our two hypotheses and simplifying we have $v \geq (t-1)\binom{n}{\lambda+1}$ with equality only if both hypotheses hold with equality.

Let $\mathcal{S}$ be the set of all $(\lambda+1)$-subsets of $\{1, \ldots, n\}$.  Let $U=\{u_{S,j} \mid S \in \mathcal{S}, j \in \{1,\ldots,t-1\}\}$ be a set of $(t-1)\binom{n}{\lambda+1}$ elements indexed by $\mathcal{S} \times \{1,\ldots,t-1\}$, and let $W$ be a set of $v-(t-1)\binom{n}{\lambda+1}$ elements that is disjoint from $U$. If $k = (t-1)\binom{n-1}{\lambda}$, let $W_1=\cdots=W_n=\emptyset$. Otherwise, we claim that there is a $\PD_{\lambda}\left(n;v-(t-1)\binom{n}{\lambda+1}, k-(t-1)\binom{n-1}{\lambda},t\right)$ on point set $W$ with block set $\mathcal{W}=\{W_1, W_2, \ldots, W_n\}$ such that $|\mathcal{W}_x| \leq \lambda$ for each $x \in W$. Such a packing design exists by Lemma~\ref{nk/v Lemma} because, using our hypothesis that $\lambda v \geq nk-(t-1)\binom{n}{\lambda+1}$,
\begin{eqnarray*}
\frac{n\left(k - (t-1)\binom{n-1}{\lambda}\right)}{|W|} &=& \frac{nk-(\lambda+1)(t-1)\binom{n}{\lambda+1}}{v-(t-1)\binom{n}{\lambda+1}} \\
&=& \frac{\left(nk - (t-1)\binom{n}{\lambda+1}\right) - \lambda (t-1)\binom{n}{\lambda+1}}{v-(t-1)\binom{n}{\lambda+1}} \\
&\leq& \frac{\lambda v - \lambda(t-1) \binom{n}{\lambda+1}}{v - (t-1)\binom{n}{\lambda+1}} \\
&=& \lambda.
\end{eqnarray*}
So the packing design $(W,\mathcal{W})$ does indeed exist.

Let $\mathcal{B} = \{B_1, \ldots, B_n\}$, where, for each $i \in \{1, \ldots, n\}$, $B_i$ is the $k$-set
\[
W_i \cup \{u_{S,j} \mid S \in \mathcal{S}, i \in S, j \in \{1,\ldots,t-1\}\}.
\]
Note that $|\mathcal{B}_x| = \lambda+1$ for all $x \in U$ and $|\mathcal{B}_x| \leq \lambda$ for all $x \in W$.  We claim $(U \cup W,\mathcal{B})$ is the required packing design.  To see this, let $T$ be a $t$-subset of $U \cup W$.  If $T \not\subseteq U$, then there is some $y \in T \cap W$ and we have $|\bigcap_{x \in T}\mathcal{B}_x| \leq |\mathcal{B}_y| \leq \lambda$.  Otherwise, $T \subseteq U$, and we must have $\{u_{S,j},u_{S',j'}\} \subseteq T$ for some $S,S' \in \mathcal{S}$ and $j,j' \in \{1,\ldots,t-1\}$ where $S \neq S'$.  So,
\[
\left|\bigcap_{x \in T}\mathcal{B}_x\right| \leq | \{B_i \mid i \in S \cap S'\}| = |S \cap S'| \leq \lambda. \qedhere
\]
\end{proof}

\begin{example}
Let $W_1,\ldots,W_4$, where $W_i=\{w_{i1},w_{i2}\}$, be the blocks of a $\PD(4;8,2)$ in which each point is in exactly one block. Applying the scheme described in the proof of Lemma~\ref{General Construction} to this packing design, we obtain a $\PD(4;14,5)$ whose blocks are
\[
\begin{array}{l}
\{w_{11}, w_{12}, u_{12}, u_{13}, u_{14} \} \\
\{w_{21}, w_{22}, u_{12}, u_{23}, u_{24} \} \\
\{w_{31}, w_{32}, u_{13}, u_{23}, u_{34} \} \\
\{w_{41}, w_{42}, u_{14}, u_{24}, u_{34} \}
\end{array}
\]
where we write $u_{xy}$ as a shorthand for $u_{\{x,y\},1}$.
\end{example}

By combining Lemmas~\ref{General Upper Bound} and~\ref{General Construction} we can prove the main  result of this section.

\begin{theorem} \label{PDN}
Let $v \geq k \geq t \geq 2$, $\lambda \geq 1$ and $n \geq 1$ be integers.  We have that $\PDN_{\lambda}(v,k,t)=n$ if
\[nk-(t-1)\binom{n}{\lambda+1} \leq \lambda v < (n+1)k- (t-1)\binom{n+1}{\lambda+1}.\]
\end{theorem}

\begin{proof}
First note that $nk-(t-1)\binom{n}{\lambda+1} < (n+1)k- (t-1)\binom{n+1}{\lambda+1}$ implies
\begin{equation}\label{E:kBig}
k > (t-1)\binom{n}{\lambda} > (t-1) \binom{n-1}{\lambda}.
\end{equation}
Since \eqref{E:kBig} and $nk-(t-1)\binom{n}{\lambda+1} \leq \lambda v$ hold, there exists a $\PD_\lambda(n;v,k,t)$ by Lemma~\ref{General Construction}.

We will complete the proof by showing that no $\PD_\lambda(n+1;v,k,t)$ exists. By Lemma~\ref{General Upper Bound}, it suffices to show that
\begin{equation}\label{E:SJBspecial}
(t-1)\binom{n+1}{\lambda+1} < v\binom{q}{\lambda+1}+r\binom{q}{\lambda}
\end{equation}
where $q$ and $r$ are the integers such that $(n+1)k=qv+r$ and $0 \leq r < v$. To determine $q$ and $r$, note that $\lambda v < (n+1)k-\binom{n+1}{2}$ implies $\lambda v < (n+1)k$ and hence $q \geq \lambda$. On the other hand, \eqref{E:kBig} yields 
$(t-1)\binom{n}{\lambda+1}<\frac{n-\lambda}{\lambda+1}k$. Substituting this into $nk-(t-1)\binom{n}{\lambda+1} \leq \lambda v$ and simplifying produces $(n+1)k<(\lambda+1)v$. Hence $q<\lambda+1$. Thus $q=\lambda$ and $r=k(n+1)-\lambda v$, so \eqref{E:SJBspecial} becomes
\[
(t-1)\binom{n+1}{\lambda+1} < k(n+1)-\lambda v,
\]
which is equivalent to our hypothesis that $\lambda v < k(n+1)-(t-1)\binom{n+1}{\lambda+1}$.
\end{proof}

We now determine the values of $k$ for which Theorem~\ref{PDN} applies and the results it gives for these values. For given values of $v$, $t$ and $\lambda$, Theorem~\ref{PDN} shows that $\PDN_{\lambda}(v,k,t)=n$ where $n$ is the smallest nonnegative integer such that
\begin{equation}\label{E:lvSmallerCritereon}
\lambda v < (n+1)k-(t-1)\binom{n+1}{\lambda+1}
\end{equation}
if such an integer exists. 
Rearranging \eqref{E:lvSmallerCritereon} to solve for $k$, Theorem~\ref{PDN} gives the value of $\PDN_{\lambda}(v,k,t)$ when
\begin{equation}\label{E:lApplicabilityBound}
k >\frac{\lambda v}{n+1} + \frac{t-1}{\lambda+1}\binom{n}{\lambda}
\end{equation}
for some positive integer $n$. As $v$ becomes large, the value of $n$ that minimizes the right hand side of \eqref{E:lApplicabilityBound} approaches $((\lambda+1)!\frac{v}{t-1})^{1/(\lambda+1)}$ and the minimum value approaches $c_{t,\lambda} 
v^{\lambda/(\lambda+1)}$ where $c_{t,\lambda} 
=(\lambda+1)(\frac{t-1}{(\lambda + 1)!})^{1/(\lambda + 1)}$. So, as $v$ becomes large for fixed $t$ and $\lambda$, Theorem~\ref{PDN} determines the value of $\PDN_{\lambda}(v,k,t)$ if and only if $k \geq f_{t,\lambda}(v)$ for some function $f_{t,\lambda}$ asymptotic to $c_{t,\lambda} 
v^{\lambda/(\lambda+1)}$.
Using a similar analysis, it can be shown that Theorem 3.1 can be applied to produce a bound only when $k \geq g_{t,\lambda}$ for some function $g_{t,\lambda}$ asymptotic to $(t-1)^{1/(\lambda+1)} v^{\lambda/(\lambda+1)},$ which differs from the asymptotic value of $f_{t,\lambda}$ only by a constant factor. Note that $f_{2,1} = \sqrt{2v}$ and $g_{2,1} = \sqrt{v}$.

Let $c>0$ and $0 \leq \alpha \leq 1$ be constants, and consider a regime in which $t$ and $\lambda$ are fixed and $k \sim cv^\alpha$ as $v$ becomes large. Assume $k \geq f_{t,\lambda}(v)$ so that Theorem~\ref{PDN} applies. By considering the dominant terms of \eqref{E:lvSmallerCritereon} in this regime we can deduce the following.
\begin{itemize}
    \item
If $\alpha = 1$, then $\PDN_{\lambda}(v,k,t)\sim \lfloor \frac{\lambda}{c} \rfloor$.
    \item
If $\frac{\lambda}{\lambda+1} < \alpha < 1$, then $\PDN_{\lambda}(v,k,t)\sim \frac{\lambda}{c}v^{1-\alpha}$.
    \item
If $\alpha = \frac{\lambda}{\lambda+1}$, then $\PDN_{\lambda}(v,k,t)\sim dv^{1-\alpha}$ where $x=d$ is the least positive real solution to the equation $\lambda=cx-\frac{t-1}{(\lambda+1)!}x^{\lambda+1}$. (That such a solution exists is guaranteed by our assumption that $k \geq f_{t,\lambda}(v)$.)
\end{itemize}

Theorem~\ref{PDN} captures most, but not all, of the situations in which Theorem~\ref{General Upper Bound} and Lemma~\ref{General Construction} combine to determine a packing number. For fixed positive integers $k$, $t$ and $\lambda$ with $k \geq t$, and nonnegative integers $n$, Theorem~\ref{PDN} requires that
\[
kn-(t-1)\binom{n}{\lambda+1} < k(n+1)-(t-1)\binom{n+1}{\lambda+1},
\]
which holds if and only if $(t-1)\binom{n}{\lambda} < k$.
We define $\ell$ to be the least integer such that $(t-1)\binom{\ell}{\lambda} > k$, and note that
Theorem~\ref{PDN} will determine values of $\PDN_{\lambda}(v,k,t)$ 
when $\lambda v < k\ell-(t-1)\binom{\ell}{\lambda+1}$ and these values will be less than $\ell$.
Since $k < (t-1)\binom{\ell}{\lambda}$, Lemma~\ref{General Construction} cannot produce packing designs with more than $\ell$ blocks.  However, it can produce packing designs with exactly $\ell$ blocks when $\lambda v \geq \ell k-(t-1)\binom{\ell}{\lambda+1}$, and in some cases these will achieve the bound of Theorem~\ref{General Upper Bound}. We encapsulate this in the following theorem.

\begin{theorem}\label{T:nBig}
Let $k \geq t \geq 2$ and $\lambda \geq 1$ be integers and let $\ell$ be the least integer such that $(t-1)\binom{\ell}{\lambda} > k$. Then $\PDN_{\lambda}(v,k,t)=\ell$ if
\begin{equation}\label{E:nBigConds}
\ell k-(t-1)\binom{\ell}{\lambda+1} \leq \lambda v < \frac{\lambda+1}{\lambda+2}(\ell+1)k- \frac{t-1}{\lambda+2}\binom{\ell+1}{\lambda+1}.
\end{equation}
\end{theorem}

\begin{proof}
Note that
\begin{equation}\label{E:n0Cons}
(t-1)\binom{\ell-1}{\lambda} \leq k
\end{equation}
by the definition of $\ell$. In view of \eqref{E:n0Cons} and the first inequality of \eqref{E:nBigConds}, Lemma~\ref{General Construction} guarantees the existence of a $\PDN_{\lambda}(\ell;v,k,t)$. We will complete the proof by showing that Theorem~\ref{General Upper Bound} can be applied with $q=\lambda+1$ to show that no $\PDN_{\lambda}(\ell+1;v,k,t)$ exists.

From \eqref{E:n0Cons} we have $(t-1)\binom{\ell}{\lambda+1} \leq \frac{1}{\lambda+1}\ell k$ and substituting this into the first inequality of \eqref{E:nBigConds} and simplifying yields $\ell k \leq (\lambda+1)v$. Thus $(\ell+1)k < (\lambda+2)v$, since $v > k$ follows from \eqref{E:n0Cons} and the first inequality of \eqref{E:nBigConds}. Now $(t-1)\binom{\ell}{\lambda} > k$ implies $(t-1)\binom{\ell+1}{\lambda+1} > \frac{1}{\lambda+1}(\ell+1)k$ and substituting this into the second inequality of \eqref{E:nBigConds} and simplifying yields $(\ell+1)k > (\lambda+1)v$. So, applying Theorem~\ref{General Upper Bound} with $q=\lambda+1$ and $r=(\ell+1)k-(\lambda+1)v$, shows that no $\PDN_{\lambda}(\ell+1;v,k,t)$ exists provided
\[
(t-1)\binom{\ell+1}{\lambda+1} < v+(\lambda+1)\bigl((\ell+1)k-(\lambda+1)v\bigr).
\]
Routine manipulation shows this is equivalent to the second inequality of \eqref{E:nBigConds}.
\end{proof}

A tedious simplification, using first $\binom{\ell+1}{\lambda+1}=\binom{\ell}{\lambda+1}+\binom{\ell}{\lambda}$ and then $(\lambda+1)\binom{\ell}{\lambda+1}=(\ell-\lambda)\binom{\ell}{\lambda}$, shows that the leftmost expression in \eqref{E:nBigConds} is less than the rightmost precisely because $(t-1)\binom{\ell}{\lambda} > k$. So, in particular, we may apply Theorem~\ref{T:nBig} when we have equality in the first inequality in \eqref{E:nBigConds}. Using this, we immediately obtain a generalization of a result in \cite{GashkovEkbergTaub}.

\begin{corollary}\label{C:GETGen}
Let $t \geq 2$, $\lambda \geq 1$ and $n \geq \lambda+1$ be integers such that $(t-1)\tbinom{n}{\lambda+1} \equiv 0 \pmod{\lambda}$.  We have \[\PDN_{\lambda}\left(\tfrac{t-1}{\lambda}\tbinom{n}{\lambda+1},(t-1)\tbinom{n-1}{\lambda},t\right)=n.\]
\end{corollary}

Specializing Corollary~\ref{C:GETGen} to the case $(t,\lambda)=(2,1)$ yields \cite[Theorem~2]{GashkovEkbergTaub}.

\section{Directed Packings} \label{sec:DPN}
In this section, we extend the results of Theorem~\ref{PDN} to directed packings $\DPD(v,k)$.  In order to do this, we first show in Lemma~\ref{lemma:directing} that the blocks of any $\PD_{2}(v,k)$ whose points all have frequency at most $3$ may be ordered in such a way to produce a $\DPD(v,k)$.

We note that Lemma~\ref{lemma:directing} applies even for packings whose blocks do not all have the same size.
For this reason, we introduce notation for packing designs with $t=2$ and blocks of varying sizes.  If $V$ is a set of $v$ points and $\cB$ is a collection of subsets of $V$ (of any sizes including 0) 
satisfying the property that any pair of points appears in at most $\lambda$ elements of $\cB$, then we say $(V,\cB)$ is a $\PD_{\lambda}(v)$.  A $\PD_{\lambda}(v)$ which has $n$ blocks is denoted by $\PD_{\lambda}(n;v)$. 
Similarly, we define a $\DPD(v)$ to be a directed packing design on $v$ points with $(t,\lambda)=(2,1)$ and blocks of any sizes, and a $\DPD(n;v)$ to be such a design with $n$ blocks. For technical reasons below, it will be important that we allow a $\PD_2(v)$ or a $\DPD(v)$ to contain empty blocks and repeated blocks.

For sequences $A=(a_1,\ldots,a_s)$ and $B=(b_1,\ldots,b_t)$ we use $A+B$ to denote the concatenated sequence $(a_1,\ldots,a_s,b_1,\ldots,b_t)$. For a sequence $A$ and a subset $S$ of the entries of $A$, we use $A[S]$ to denote the subsequence of $A$ that contains only the elements of $S$.

\begin{lemma} \label{lemma:directing}
Let $(V,\mathcal{B})$ be a $\PD_2(n;v)$
such that $r_{\mathcal{B}}(x)
\leq 3$ for each $x \in V$. There is a $\DPD(n;v)$,
$(V,\{T_B\}_{B \in \mathcal{B}})$, such that $T_B$ is a permutation of $B$ for each $B \in \mathcal{B}$.
\end{lemma}

\begin{proof}
We proceed by induction on $v$.  The result is trivial when $v=1$.  Now suppose $(V,\mathcal{B})$ is a $\PD_2(v)$ with $v \geq 2$ and $|\mathcal{B}_x| \leq 3$ for each $x \in V$.  Let $a \in V$, and let $\mathcal{B}=\{B_1, \ldots, B_n\}$ and $\mathcal{B}_a = \{B_1, \ldots, B_{t}\}$, where $t = r_{\cB}(a) \in \{0,1,2,3\}$.
Let $B'_i=B_i \setminus \{a\}$ for each $i \in \{1, \ldots, n\}$.
By induction, there is a $\DPD(v-1)$, $(V\setminus\{a\}, \{T_1', \ldots, T_n'\})$, such that $T_i'$ is a permutation of $B'_i$ for each $i \in \{1, \ldots, n\}$.

We will create a $\DPD(v)$, $(V,\mathcal{T})$, with the required properties where $\mathcal{T}=\{T_1, \ldots, T_n\}$, $T_i=T_i'$ for every $i \in \{t+1, \ldots, n\}$ and for $i \in \{1, \ldots, t\}$, $T_i$ is a permutation of $B_i = B_i' \cup \{a\}$ such that $T_i'$ is a subsequence of $T_i$.  In other words, for $i \in \{1, \ldots, t\}$, $T_i$ is obtained by inserting $a$ into $T_i'$.  Note that to confirm that $(V,\mathcal{T})$ is a $\DPD(v)$, it will suffice to confirm that, for each point $u \in V \setminus\{a\}$ such that $|\mathcal{B}_u \cap \mathcal{B}_a| = 2$, we have that the sequence $(u,a)$ is a subsequence of one of the blocks in $\mathcal{T}_u \cap \mathcal{T}_a$ and $(a,u)$ is a subsequence of the other.

We specify how $a$ should be inserted into $T_1, \ldots, T_t$ according to the value of $t$.  Note that if $t=0$ there is nothing to do, and if $t=1$, then we can take $T_1 = (a) + T_1'$.  If $t=2$, we take $T_1 = (a)+T_1'$ and $T_2 = T_2'+(a)$.  In each of these cases, it is easy to see that $\mathcal{T}$ forms the blocks of a $\DPD(v)$ with the required properties.

We now consider the case that $t=3$.  Let $X=B_1' \cap B_2'$, $Y=B_1' \cap B_3'$ and $Z=B_2' \cap B_3'$.  Note that $X$, $Y$ and $Z$ are pairwise disjoint because $\mathcal{B}_a = \{B_1,B_2,B_3\}$ and $(V,\mathcal{B})$ is a $\PD_2(v)$. Let $p=|X|$, $q=|Y|$ and $r=|Z|$ (noting that some or all of $p$, $q$ and $r$ may be zero) and say $X=\{x_1, \ldots, x_p\}$, $Y=\{y_1, \ldots, y_p\}$ and $Z=\{z_1, \ldots, z_p\}$.  Suppose without loss of generality that:
\begin{itemize}
\item $(x_1, \ldots, x_p)$ and $(y_1, \ldots, y_q)$ are subsequences of $T_1'$,
\item $(x_p, \ldots, x_1)$ and $(z_1, \ldots, z_r)$ are subsequences of $T_2'$, and
\item $(y_q, \ldots, y_1)$ and $(z_r, \ldots, z_1)$ are subsequences of $T_3'$.
\end{itemize}

Below, it will be convenient to use $x_0, x_{p+1}, y_0, y_{q+1}, z_0, z_{r+1}$ as placeholders representing the beginning or end of certain sequences.  For $i \in \{0, \ldots, p+q\}$, define $j(i)$, $k(i)$, $R(i)$ and $S(i)$ as follows.

\begin{itemize}
\item Let $j(i)$ be the greatest element $j$ of $\{1, \ldots, p\}$ such that $x_j$ is one of the first $i$ entries of $T_1'[X \cup Y]$ if such an element exists; otherwise, let $j(i)=0$.

\item Let $k(i)$ be the greatest element $k$ of $\{1, \ldots, q\}$ such that $y_k$ is one of the first $i$ entries of $T_1'[X \cup Y]$ if such an element exists; otherwise, let $k(i)=0$.
\item
Let $R(i)$ be the set of consecutive integers in $\{0, \ldots, r+1\}$ such that, in $(z_0,x_{p+1})+T_2'[X \cup Z] + (x_0,z_{r+1})$, $\min(R(i))$ is the largest index of an element of $Z$ that occurs before $x_{j(i)+1}$ and $\max(R(i))$ is the smallest index of an element of $Z$ that occurs after $x_{j(i)}$.  So $R(i)$
contains the set of all indices of elements of $Z$ that lie between $x_{j(i)}$ and $x_{j(i)+1}$ in this block, as well as $\min(R(i))$ and $\max(R(i))$.

\item
Let $S(i)$ be the set of consecutive integers in $\{0, \ldots, r+1\}$ such that, in $(z_{r+1},y_{q+1})+T_3'[Y \cup Z] + (y_0,z_{0})$, $\max(S(i))$ is the smallest index of an element of $Z$ that occurs before $y_{k(i)+1}$ and $\min(S(i))$ is the largest index of an element of $Z$ that occurs after $y_{k(i)}$.  So $S(i)$
contains the set of all indices of elements of $Z$ that lie between $y_{k(i)}$ and $y_{k(i)+1}$ in this block, as well as $\min(S(i))$ and $\max(S(i))$.
\end{itemize}
Note that $R(i)$ and $S(i)$ are well defined because $x_{j(i)}$ and $x_{j(i+1)}$ occur between $z_0$ and $z_{r+1}$ in $(z_0,x_{p+1})+T_2'[X \cup Z] + (x_0,z_{r+1})$ and $y_{k(i)}$ and $y_{k(i+1)}$ occur between $z_{r+1}$ and $z_0$ in $(z_{r+1},y_{q+1})+T_3'[Y \cup Z] + (y_0,z_{0})$.
We note the following properties.
\begin{enumerate}
\item[(i)] $|R(i)| \geq 2$ and $|S(i)| \geq 2$ for each $i \in \{0, \ldots, p+q\}$.

This is because $\min(R(i))<\max(R(i))$ and $\min(S(i))<\max(S(i))$ by the definitions of $R(i)$ and $S(i)$.

\item[(ii)] $j(0)=k(0)=0$, so $\max(R(0)) = r+1$ and $\min(S(0))=0$.

We have $j(0)=k(0)=0$ by definition.  Thus $z_{r+1}$ is the first element of $Z$ after $x_{j(0)}$ in $(z_0,x_{p+1})+T_2'[X \cup Z] + (x_0,z_{r+1})$ and $z_0$ is the first element of $Z$ after $y_{k(0)}$ in $(z_{r+1},y_{q+1})+T_3'[Y \cup Z] + (y_0,z_{0})$.

\item[(iii)] $j(p+q)=p$ and $k(p+q)=q$, so $\min(R(p+q))=0$ and $\max(S(p+q))=r+1$.

The first $p+q$ elements of $T_1'[X\cup Y]$ comprise all of $X$ and $Y$, so $j(p+q)=p$ and $k(p+q)=q$ follow.  Thus $\min(R(p+q))$ is the index of the last element of $Z$ before $x_{p+1}$ in $(z_0,x_{p+1})+T_2'[X \cup Z] + (x_0,z_{r+1})$ and $\max(S(p+q))$ is the index of the last element of $Z$ before $y_{q+1}$ in $(z_{r+1},y_{q+1})+T_3'[Y \cup Z] + (y_0,z_{0})$.

\item[(iv)] If the $(i+1)$st entry of $T_1'[X \cup Y]$ is in $X$, then $j(i+1)=j(i)+1$, $k(i+1)=k(i)$, $\max(R(i+1))=\min(R(i))+1$ and $S(i+1)=S(i)$.

Since the $(i+1)$st entry
of $T_1'[X \cup Y]$ is in $X$, it is clear that $j(i+1)=j(i)+1$. Since the $(i+1)$st entry is not in $Y$, $k(i+1)=k(i)$, and it follows that $S(i+1)=S(i)$.

Also, $\max(R(i+1))=\min(R(i))+1$ because $x_{j(i)+1}=x_{j(i+1)}$ and, in $(z_0,x_{p+1})+T_2'[X \cup Z] + (x_0,z_{r+1})$, the indices of the last element of $Z$ before $x_{j(i+1)}$ and the first element of $Z$ after $x_{j(i+1)}$ are consecutive.

\item[(v)] If the $(i+1)$st entry
of $T_1'[X \cup Y]$ is in $Y$, then $j(i+1)=j(i)$, $k(i+1)=k(i)+1$, $R(i+1)=R(i)$, and $\min(S(i+1))=\max(S(i))-1$.

This follows similarly to (iv), with the roles of $X$ and $Y$ interchanged, and considering the block $T_3'$ in place of $T_2'$.
\end{enumerate}

Let $\ell$ be the least element of $\{0, \ldots, p+q\}$ such that $\min(R(\ell)) < \max(S(\ell))$.  (Note that such an $\ell$ exists by (iii).)
\vspace*{1ex}

\noindent
{\bf Claim.} $|R(\ell) \cap S(\ell)| \geq 2$.
\vskip 1ex

{\em Proof of claim.}
Note that $\ell \neq 0$ since $\min(R(0)) \geq 0 = \max(S(0))$ by (ii).  Thus $\ell-1 \geq 0$ and by minimality of $\ell$, $\min(R(\ell-1)) \geq \max(S(\ell-1))$.

First, suppose that the $\ell$th entry
of $T_1'[X \cup Y]$ is in $X$.  Then by (iv),
\[
\max(R(\ell)) = \min(R(\ell-1))+1 \geq \max(S(\ell-1))+1 = \max(S(\ell))+1.
\]
Thus, since $\min(R(\ell)) < \max(S(\ell))$, it follows that
\[
\min(R(\ell)) \leq \max(S(\ell))-1 < \max(S(\ell)) < \max(R(\ell)).
\]
As the elements of $R(\ell)$ are consecutive, this means that $R(\ell)$ must contain both $\max(S(\ell))-1$ and $\max(S(\ell))$.  But from (i), $|S(\ell)| \geq 2$, and the elements of $S(\ell)$ are consecutive, so $\{\max(S(\ell))-1, \max(S(\ell))\} \subseteq S(\ell)$.  Thus $|R(\ell) \cap S(\ell)| \geq 2$.

Otherwise, the $\ell$th entry
of $T_1'[X \cup Y]$ is in $Y$.  A similar argument using (v) shows that $|R(\ell) \cap S(\ell)| \geq 2$.  This concludes the proof of the claim.
\vskip 1ex

It follows that for some $m$, $\{m, m+1\} \subseteq R(\ell) \cap S(\ell)$.

\begin{itemize}
\item Let $T_1$ be the permutation of $B_1$ obtained by inserting $a$ into $T_1'$ immediately after the $\ell$th entry
(and at the beginning of the sequence if $\ell=0$).  Note that, in $(x_0,y_0) + T_1 + (x_{p+1},y_{q+1})$, $a$ will appear between $x_{j(\ell)}$ and $x_{j(\ell)+1}$ and also between $y_{k(\ell)}$ and $y_{k(\ell)+1}$.
\item Let $T_2$ be a permutation of $B_2$ obtained by inserting $a$ into $T_2'$ so that, in $(z_0,x_{p+1})+T_2+(x_0,z_{r+1})$, $a$ appears between $x_{j(\ell)+1}$ and $x_{j(\ell)}$ and also between $z_m$ and $z_{m+1}$.  A suitable position to insert $a$ exists because $\{m,m+1\} \subseteq R(\ell)$.
\item Let $T_3$ be a permutation of $B_3$ obtained by inserting $a$ into $T_3'$ so that, in $(z_{r+1},y_{q+1})+T_3'+(y_0,z_0)$, $a$ appears between $y_{k(\ell)+1}$ and $y_{k(\ell)}$ and also between $z_{m+1}$ and $z_m$.  A suitable position to insert $a$ exists because $\{m,m+1\}\subseteq S(\ell)$.
\end{itemize}

From the above, it is easy to confirm that, for each point $u \in V \setminus \{a\}$ such that $|\mathcal{B}_u \cap \mathcal{B}_a|=2$ we have that $(u,a)$ is in one of the blocks in $\mathcal{T}_u \cap \mathcal{T}_a$ and $(a,u)$ is in the other.
\end{proof}

\begin{example}
We illustrate the inductive step of the proof of Lemma~\ref{lemma:directing} by constructing a $\DPD(4;12,7)$.  The following blocks form a $\PD_2(4;12,7)$ on point set $\{0, \ldots, 11\}$.
\begin{eqnarray*}
B_1&=& \{0,1,2,3,4,5,6\} \\
B_2&=& \{0,1,3,4,7,8,9\} \\
B_3&=& \{0,2,5,6,9,10,11\} \\
B_4&=& \{1,2,7,8,9,10,11\}
\end{eqnarray*}
Removing the point $a=0$, we can order the remaining elements of each block to form the following $\DPD(4;11)$ with block sizes six and seven,
\begin{eqnarray*}
T_1'&=& (1,2,3,4,5,6) \\
T_2'&=& (4,3,7,8,9,1) \\
T_3'&=& (6,5,10,11,9,2) \\
T_4'&=& (2,1,9,11,10,8,7)
\end{eqnarray*}
We have $X=(B_1 \cap B_2) \setminus\{0\}=\{1,3,4\}$, $Y=(B_1 \cap B_3) \setminus\{0\}=\{2,5,6\}$ and $Z=(B_2 \cap B_3) \setminus\{0\}=\{9\}$, where $(x_1,x_2,x_3)=(1,3,4)$, $(y_1,y_2,y_3)=(2,5,6)$ and $(z_1)=(9)$.  It is straightforward to verify that:
\begin{itemize}
\item $j(0)=0$, $j(1)=j(2)=1$, $j(3)=2$, $j(4)=j(5)=j(6)=3$;
\item $k(0)=k(1)=0$, $k(2)=k(3)=k(4)=1$, $k(5)=2$, $k(6)=3$;
\item $R(0)=\{1,2\}$, $R(1)=R(2)=\{0,1,2\}$, $R(3)=R(4)=R(5)=R(6)=\{0,1\}$;
\item $S(0)=S(1)=\{0,1\}$, $S(2)=S(3)=S(4)= \{0,1,2\}$, $S(5)=S(6)=\{1,2\}$.
\end{itemize}
Thus $\ell=1$, and $\{0,1\} \subseteq R(1) \cap S(1)$.  We form $T_1$ by inserting $0$ after the first entry of $T_1'$.   To form $T_2$ from $T_2'$, we insert $0$ so that, in $(z_0,x_4) + T_2' + (x_0,z_2)$, it is between $x_{j(1)+1}=x_2=3$ and $x_{j(1)}=x_1=1$ and also between $z_0$ and $z_1=9$ (i.e.\ $0$ must appear before $9$ in $T_2$).  Similarly, we insert $0$ into $T_3'$ so that, in $(z_2,y_4) + T_2' + (y_0,z_0)$, it is between $y_{k(1)+1}=y_1=2$ and $y_{k(1)}=y_0$ and also between $z_1=9$ and $z_0$.
Also taking $T_4=T_4'$, we obtain a $\DPD(4;12,7)$ with the following blocks.

\begin{eqnarray*}
T_1 &=& (1,0,2,3,4,5,6) \\
T_2 &=& (4,3,7,8,0,9,1) \\
T_3 &=& (6,5,10,11,9,2,0) \\
T_4 &=& (2,1,9,11,10,8,7).
\end{eqnarray*}
\end{example}

\begin{theorem} \label{DPDN}
Let $v \geq k \geq 2$ and $n \geq 1$ be integers.
If
$nk-\binom{n}{3} \leq 2v < (n+1)k-\binom{n+1}{3}$,
then $\DPDN(v,k)=n=\PD_2(v,k)$.
\end{theorem}

\begin{proof}
The upper bound follows from Lemmas~\ref{directed vs undirected} and~\ref{General Upper Bound}. By Lemma~\ref{General Construction}, there exists a $\PD_2(n;v,k)$, $(V,\mathcal{B})$, such that $|\mathcal{B}_x| \leq 3$ for each $x \in V$; existence of a $\DPD(n;v,k)$ now follows from Lemma~\ref{lemma:directing}.
\end{proof}

\section{Summary and Discussion} \label{Sec:Conclusion}

For given $t$ and $\lambda$, Theorem~\ref{PDN} completely characterizes the packing number $\PDN_{\lambda}(v,k,t)$ when the block size, $k$, is large with respect to $v$. Similarly, for given $\lambda$, Theorem~\ref{DPDN} characterizes the directed packing number $\DPDN(v,k)$ when $k$ is large with respect to $v$. 
Specifically, Theorem~\ref{PDN} shows that $\PDN_{\lambda}(v,k,t)=n$ when $n$ is an integer satisfying
\begin{equation}\label{condition}
nk-(t-1)\binom{n}{\lambda+1} \leq \lambda v < (n+1)k-(t-1)\binom{n+1}{\lambda+1}.
\end{equation}
This theorem applies when, for given integers $v$, $k$ and $\lambda$,
\[
k > \frac{\lambda v}{n+1} + \frac{t-1}{\lambda+1}\binom{n}{\lambda}
\]
for some integer $n$.  Theorem~\ref{DPDN} shows that $\DPDN(v,k)=\PDN_2(v,k)$ when $v$ and $k$ satisfy the same constraints for $(t,\lambda)=(2,2)$.

The upper bound for Theorem~\ref{PDN} is obtained from Theorem~\ref{General Upper Bound}, so for any parameter sets to which it applies, it establishes the bound of Theorem~\ref{General Upper Bound} is tight. In particular, when $\lambda=1$, Theorem~\ref{General Upper Bound} specializes to the stronger form of the second Johnson bound and hence shows that this bound is tight. Of course, unlike the second Johnson bound, the bound in Theorem~\ref{General Upper Bound} applies for arbitrary $\lambda$.

For arbitrary $\lambda$, it is natural to ask how close the previously known general upper bounds on $\PDN_{\lambda}(v,k,t)$ are to the actual value of the packing number. We can now answer this question with respect to the range of interest. The Johnson-Sch\"{o}nheim bound $U_{\lambda}(v,k,t)$ is known to perform relatively poorly when $k$ is large with respect to $v$. For instance, when $(t,\lambda)=(2,1)$, it is noted in~\cite{Stinson} that $U_1(v,k,2)$ is a good bound when $v \geq k^2-k+1$ but the second Johnson bound is better otherwise. 

Our results bear out this observation. As in our discussion after Theorem~\ref{General Upper Bound}, we let $c>0$ and $0 \leq \alpha \leq 1$ be constants, and consider a regime in which $\lambda$ is fixed and $k \sim cv^\alpha$ as $v$ becomes large. Again, we assume $k \geq f_{t,\lambda}(v)$ so that Theorem~\ref{PDN} applies.

\begin{itemize}

    \item
If $\frac{\lambda}{\lambda+1} \leq \alpha < 1$, then we have
\[U_{\lambda}(v,k,t) \sim \frac{\lambda v^t}{k^t} \sim \frac{\lambda v^{t(1-\alpha)}}{c^t}.\]
However, in these cases we have seen that $\PDN_{\lambda}(v,k,t)\sim dv^{1-\alpha}$ for a constant $d$ depending on $\lambda$ and $c$ (in fact $d=\frac{\lambda}{c}$ if $\frac{\lambda}{\lambda+1} < \alpha < 1$). So the order of $U_{\lambda}(v,k,t)$ is the $t$th power of the order of the true packing number.
    \item
If $\alpha = 1$, then the asymptotic behaviour of $U_{\lambda}(v,k,t)$ is complicated by the floors. However, we certainly have that
\begin{multline}U_{\lambda}(v,k,t) \geq \frac{v}{k}\left(\frac{v}{k}\left(\cdots\left(\frac{v}{k}\left(\frac{\lambda v}{k}-1\right)-1\right)\cdots\right)-1\right)-1 \\
\sim \frac{\lambda}{c^t}-\frac{1}{c^{t-1}}-\frac{1}{c^{t-2}}-\cdots-\frac{1}{c}-1.
\end{multline}
In this case we have seen $\PDN_{\lambda}(v,k,t)\sim \lfloor \frac{\lambda}{c} \rfloor$. So $U_{\lambda}(v,k,t)$ will be far from a tight bound when $c$ is small.
\end{itemize}

We note that the behaviour above contrasts sharply with the case where $k$ is fixed as $v$ grows. For instance, it is known that if $k \in \{3,4\}$ and $v \geq 20$, $U_1(v,k,t)-\PDN(v,k) \leq 1$, see~\cite{MillsMullin}. In fact, Theorem~\ref{PDN} explains two of the small cases in which $\PDN(v,4) < U_1(v,4,2)=B(v,4)$,
namely that $\PDN(8,4)=2$ and $\PDN(9,4)=3$.

Recall that a $\PD(v,k,t)$ is equivalent to a binary code with length $v$, minimum distance $2(k-t+1)$ and constant weight $k$. The maximum size of a binary code of length $n$, minimum distance $d$ and constant weight $w$ has been well studied in the coding theory literature and is often denoted $A(n,d,w)$ (see, for example,~\cite{AgrellVardyZeger}). Thus, our results determining $\PDN_\lambda(v,k,t)$, when specialized to the case $\lambda=1$, equivalently determine $A(v,2(k-t+1),k)$ for large $k$.

As we discussed, a $\DPD(v,k,t)$, yields a $(k-t)$-insertion/deletion code of length $k$ over an alphabet of size $v$. If the definition of such a code forbids codewords containing symbols more than once, then the two objects are equivalent. So, in particular, the values of $\DPDN(v,k)$ given by Theorem~\ref{DPDN} are also the maximum sizes of $(k-2)$-insertion/deletion codes of length $k$ over an alphabet of size $v$ that forbid repeated symbols. In fact, this also determines the maximum sizes of the corresponding codes allowing repeated symbols, because it is not too difficult to see that a maximum-sized code allowing repeats can be obtained by adding the $v$ `constant' codewords to a maximum-sized code forbidding repeats. Note that $d=k-2$ is the largest choice of $d$ for which $d$-insertion/deletion codes of length $k$ are interesting objects.

Our results show that when $k$ is large with respect to $v$, $\DPDN(v,k) = \PDN_2(v,k)$.  This assertion stems from Lemma~\ref{lemma:directing}, which proves that if no element has frequency greater than three, then the blocks of a $\PD_2(v,k)$ can be directed to produce a $\DPD(v,k)$.  Directing algorithms have previously been studied for balanced incomplete block designs.  Although the blocks of every BIBD$(v,3,2\lambda)$ can be ordered to form a DBIBD$(v,3,\lambda)$ \cite{ColbournColbourn,HarmsColbourn,HarmsColbourn2},
it is known that not every $2$-fold design can be directed to form a directed design. The smallest known example of a BIBD$(v,k,2\lambda)$ whose blocks cannot be ordered to form a DBIBD$(v,k,\lambda)$ is a BIBD$(28,7,2)$, see~\cite{BennettMahmoodi}.  Additionally, in~\cite{Mahmoodi}, it was shown that none of the (finitely many) symmetric BIBD$(v,k,2)$ with $v \geq 37$ known at that time are directable.
In particular, there are exactly four non-isomorphic symmetric BIBD$(37,9,2)$~\cite{Assmus}, none of which are directable~\cite{Mahmoodi}.  Hence $\DPDN(37,9) \neq \PDN_2(37,9)$.  
Nevertheless, the equality $\DPDN(v,k)=\PDN_2(v,k)$ holds for the small values of $k$ which have been investigated in the literature, namely all cases with $k \in \{3,4,5\}$ for which both $\DPDN(v,k)$ and $\PD_2(v,k)$ are known; see~\cite{BennettMahmoodi}.  
This observation leads to the following question: 
\begin{question} \label{Conj:DPDN}
For which integers $v \geq k \geq 3$ is it true that $\DPDN(v,k) = \PDN_2(v,k)$?
\end{question}
\noindent
By the above discussion, the answer is negative for $(v,k)=(37,9)$, but this seems to be a rare occurrence.  We thus conjecture that there is only a small number of pairs $(v,k)$ which do not satisfy the equation in Question~\ref{Conj:DPDN}. 

\section{Acknowledgements}

Burgess and Danziger acknowledge support from NSERC Discovery grants RGPIN-2025-04633 and RGPIN-2022-03816, respectively.  Horsley was supported by Australian Research Council grants DP220102212 and DP240101048.


\begin{thebibliography}{99}

\bibitem{AbelEtAl}
R.~J.~R.\ Abel, A.~M.\ Assaf, I.\ Bluskov, M.\ Greig and N.\ Shalaby, New results on GDDs, covering, packing and directable designs with block size 5, {\em J.\ Combin.\ Des.} {\bf 18} (2010), 337--368.

\bibitem{AbelBennett}
R.~J.~R.\ Abel and F.~E.\ Bennett, Existence of directed BIBDs with block size 7 and related perfect 5-deletion-correcting codes of length 7, {\em Des.\ Codes Cryptogr.} {\bf 57} (2010), no.~3, 383--397. 

\bibitem{AgrellVardyZeger} 
E.~Agrell, A.~Vardy and K.~Zeger.  Upper bounds for constant-weight codes.  {\em IEEE Trans.\ Inform.\ Theory} {\bf 46} (2000), 2373--2395.

\bibitem{AssafShalabyYin2}
A.~M.\ Assaf, N.\ Shalaby and J.\ Yin, Directed packings with block size $5$, {\em Australas.\ J.\ Combin.} {\bf 17} (1998), 235--256. 

\bibitem{AssafShalabyYin}
A.~M. Assaf, N. Shalaby and J. Yin, Directed packing and covering designs with block size four, {\em Discrete Math.} {\bf 238} (2001), no.~1-3, 3--17. 

\bibitem{AssafShalabyMahmoodiYin}
A.~M. Assaf, N.\ Shalaby, A.\ Mahmoodi, J.\ Yin, Directed packings with block size $5$ and odd $v$, {\em Australas.\ J.\ Combin.} {\bf 13} (1996), 217--226. 

\bibitem{Assmus}
E.~F.\ Assmus Jr., J.~A.\ Mezzaroba and C.~J.\ Salwach.  Planes and biplanes.  In: Higher Combinatorics (Proc.\ NATO Advanced Study Inst., Berlin, 1976), NATO Adv.\ Study Ser.\ C: Math.\ and Phys.\ Sci. {\bf 31} (1977), 205--212.

\bibitem{Bennet}
F.~E. Bennett et al., Existence of DBIBDs with block size six, Utilitas Math. {\bf 43} (1993), 205--217; MR1220696.

\bibitem{BennettMahmoodi} 
F.~E.\ Bennett and A.\ Mahmoodi, Directed designs.  In: C.J.\ Colbourn and J.H.\ Dinitz (Eds).
\newblock {\em The CRC Handbook of Combinatorial Designs}.
\newblock 2nd ed. CRC Press Series on Discrete Mathematics, Boca Raton, 2007.  Section VI.20, pp.\ 441--444.

\bibitem{Bours}
P.A.H.\ Bours, On the construction of perfect deletion-correcting codes using design theory, {\em Des. Codes Cryptogr.} {\bf 6}, 5–20 (1995).

\bibitem{ChangGiovanni1}
Y.\ Chang,  G.\ Lo-Giovanni, The existence of directed BIBDs, {\em Discrete Math.} {\bf 272} (2003) 155--169.

\bibitem{ChangGiovanni2}
Y.\ Chang, G.\ Lo-Giovanni, Existence of DBIBDs with block size not divided by four, {\em Discrete Math.} {\bf 222} (2000) 27--40.

\bibitem{ColbournColbourn}
C.J.\ Colbourn and M.J.\ Colbourn.  Every twofold triple system can be directed, {\em J.\ Combin.\ Theory Ser.\ A} {\bf 34} (1983), 375--378.

\bibitem{Handbook}
C.J.\ Colbourn and J.H.\ Dinitz (Eds).
\newblock {\em The CRC Handbook of Combinatorial Designs}.
\newblock 2nd ed. CRC Press Series on Discrete Mathematics, Boca Raton, 2007.


\bibitem{GashkovEkbergTaub}
I. Gashkov, J.A.O. Ekberg and D. Taub,
A geometric approach to finding new lower bounds of $A(n,d,w)$,
{\em Des. Codes Cryptogr.} {\bf 43} (2007), 85--91.

\bibitem{Hanani}
H.\ Hanani, {Balanced incomplete block designs and related designs}, {\em Disc. Math.} {\bf 11} (1975) 255-369.

\bibitem{HarmsColbourn} J.J.\ Harms and C.J.\ Colbourn, Partitions into directed triple systems, {\em Ars Combin.} {\bf 16} (1983), 21--25.

\bibitem{HarmsColbourn2}
J.J.\ Harms and C.J.\ Colbourn, An optimal algorithm for directing triple systems using Eulerian circuits, {\em Cycles in Graphs}, Ann.\ Discrete Math.\ {\bf 27}, North-Holland Publishing Co., Amsterdam, 1985, pp.\ 433--438.

\bibitem{horsley_bounds}
D.\ Horsley, {Generalising Fisher’s inequality to coverings and packings}, {\em Combinatorica} {\bf 37} (2017) 673-696.

\bibitem{Johnson}
S.M.\ Johnson.  A new upper bound for error-correcting codes.  {\em IRE Trans.\ Inform.\ Theory} {\bf 8} (1962) 203--207.

\bibitem{Klein}
A.\ Klein. On perfect deletion-correcting codes, {\em J. Combin. Des.} {\bf 12}, (2004) 72--77.

\bibitem{LevenshteinRussian}
V.~I.\ Levenshtein.\ Binary codes capable of correcting deletions, insertions and reversals.  {\em Dokl.\ Akad.\ Nauk SSSR} {\bf 163}, 845--848.  Translation in {\em Soviet Physics Dokl.} {\bf 10} (1965), 707--710.

\bibitem{Levenshtein}
V.\ I.\ Levenshtein. On perfect codes in deletion and insertion metric, {\em Discr.\ Math.\ Appl.}, {\bf 2} (1992), 241--258.

\bibitem{LevenshteinHandbook}
V.~I.\ Levenshtein.  Deletion-correcting codes.  In: C.J.\ Colbourn and J.H.\ Dinitz (Eds).
\newblock {\em The CRC Handbook of Combinatorial Designs}.
\newblock 2nd ed. CRC Press Series on Discrete Mathematics, Boca Raton, 2007. Section VI.14, pp.\ 385--388.

\bibitem{Mahmoodi}
A.\ Mahmoodi.  Combinatorial and algorithmic aspects of directed designs, Ph.D.\ thesis, University of Toronto, 1996.

\bibitem{MillsMullin}
W.~H.\ Mills and R.~C.\ Mullin.  Coverings and Packings.  In: J. H. Dinitz, D. R. Stinson (Eds), {\em Contemporary Design Theory. A collection of surveys}, Wiley, 1992.  Chapter 9, pp.\ 371--399.

\bibitem{Schonheim}
J.\ Sch\"{o}nheim.  On maximal systems of $k$-tuples,  {\em Studia Sci.\ Math.\ Hungar.}, {\bf 1} (1966) 363--368.

\bibitem{SeberrySkillicorn}
J.~R.\ Seberry and D.~B.\ Skillicorn, All directed BIBDs with $k=3$\ exist, {\em J.\ Combin.\ Theory Ser.\ A} {\bf 29} (1980), 244--248. 

\bibitem{ShalabyYin}
N.\ Shalaby and J.\ Yin, Directed packings with block size $5$ and even $v$, {\em Des. Codes Cryptogr.} {\bf 6} (1995), 133--142. 

\bibitem{Skillicorn k=4}
D.~B.\ Skillicorn, Directed packings of pairs into quadruples, {\em J.\ Austral.\ Math.\ Soc.\ Ser.\ A} {\bf 33} (1982), 179--184. 

\bibitem{Skillicorn k=3}
D.~B.\ Skillicorn, A note on directed packings of pairs into triples, {\em Ars Combin.} {\bf 13} (1982), 227--229. 

\bibitem{StreetSeberry}
D.~J. Street and J.~R. Seberry, All DBIBDs with block size four exist, Utilitas Math. {\bf 18} (1980), 27--34. 

\bibitem{StreetWilson}
D.~J.\ Street and W.~H.\ Wilson, On directed balanced incomplete block designs with block size five, {\em Utilitas Math.} {\bf 18} (1980), 161--174. 

\bibitem{Stinson}
D.~R.\ Stinson.  An explicit formulation of the second Johnson bound.  {\em Bull.\ Inst.\ Combin.\ Appl.\ } {\bf 8} (1993), 86--92.

\bibitem{StinsonWeiYin}
D.~R.\ Stinson, R.\ Wei and J.\ Yin.  Packings.  In: C.J.\ Colbourn and J.H.\ Dinitz (Eds).
\newblock {\em The CRC Handbook of Combinatorial Designs}.
\newblock 2nd ed. CRC Press Series on Discrete Mathematics, Boca Raton, 2007.  Section VI.40, pp.\ 550--556.

\bibitem{Tonchev}
V.~D.\ Tonchev.  Codes.  In: C.J.\ Colbourn and J.H.\ Dinitz (Eds).
\newblock {\em The CRC Handbook of Combinatorial Designs}.
\newblock 2nd ed. CRC Press Series on Discrete Mathematics, Boca Raton, 2007.  Section VII.1, pp.\ 677--702.

\bibitem{WangYin}
J.\ Wang, J.\ Yin. Constructions for perfect 5-deletion-correcting codes of length 7. {\em IEEE Trans.\ Inform.\ Theory.\ } {\bf 52}, 3676--3685 (2006).
\end{thebibliography}
\end{document}